\documentclass[11pt]{amsart}
\usepackage{amsmath,amssymb,mathrsfs,enumerate}
\usepackage[bookmarks=true]{hyperref}
\newtheorem{theorem}{Theorem}[section]
\newtheorem{proposition}[theorem]{Proposition}
\newtheorem{corollary}[theorem]{Corollary}
\newtheorem{conjecture}[theorem]{Conjecture}
\newtheorem{remark}[theorem]{Remark}
\newtheorem{lemma}[theorem]{Lemma}

\theoremstyle{definition}

\newcommand{\cF}{\mathcal{F}}

\newcommand{\bR}{\mathbb{R}}

\newcommand{\Ric}{\text{Ric}}

\newcommand{\p}{\partial}

\newcommand{\gd}{\nabla}

\newcommand{\lp}{\Delta}

\DeclareMathOperator{\vol}{vol}
\newcommand{\nc}{\newcommand}
\nc{\on}{\operatorname}
\nc{\ve}{\varepsilon}
\nc{\area}{\on{Area}}
\nc{\tr}{\on{{tr}}}

\begin{document}
\title[Fill-Ins of Tori]{Fill-Ins of Tori with Scalar Curvature Bounded from Below}
\author{Yipeng Wang}
\address{Columbia University \\ 2990 Broadway \\ New York NY 10027 \\ USA}
\begin{abstract}
Let $\gamma$ be a Riemannian metric on $\Sigma = S^1 \times T^{n-2}$, where $3 \leq n \leq 7$. Consider $\Omega = B^2 \times T^{n-2}$ with boundary $\partial \Omega = \Sigma$, and let $g$ be a Riemannian metric on $\Omega$ such that the scalar curvature $R_g \geq -n(n - 1)$ and $g|_{\partial \Omega} = \gamma$. Assuming the mean curvature of $\partial \Omega$ with respect to the outward normal is positive, we establish that the total mean curvature of $\partial \Omega$ is bounded from above by a constant depending only on $n$ and $\gamma$. Furthermore, we compute the sharp constant for this estimate when $\gamma$ is a flat metric. This result resolves a special case of a conjecture by Gromov concerning total mean curvature of fill-in with scalar curvature bounded from below. The proof combines techniques developed by Shi--Tam, Shi--Wang--Wei, as well as recent work by Brendle--Hung on the systolic inequality.
\end{abstract}

\maketitle
\section{Introduction}
On a compact Riemannian manifold with boundary, a fundamental problem is understanding the relationship between the intrinsic curvature and the extrinsic curvature of the boundary. One classical result illustrating this principle is the following theorem by Shi and Tam:
\begin{theorem}[Shi-Tam \cite{Shi-Tam}]{\label{thm:Shi-Tam}}
    Let $\Omega$ be a strictly convex domain in $\bR^n$ equipped with a Riemannian metric $g$. Suppose 
    \begin{itemize}
        \item The scalar curvature $R_g\ge 0$ in $\Omega$.
        \item The induced metric on $\p\Omega$ agrees with the restriction of the Euclidean metric on $\p\Omega$: $g|_{\p \Omega}=g_{\bR^n}|_{\p\Omega}$. 
        \item The mean curvature of $\p\Omega$ with respect to $g$ is positive.
    \end{itemize}
    Then
    \[
    \int_{\p\Omega}(H_g-H_0)d\vol_g\le 0, 
    \]
    where $H_g$ and $H_0$ denotes the mean curvature of $\p\Omega$ with respect to the metric $g$ and $g_{\bR^n}$ respectively.
\end{theorem}
The Shi-Tam estimate is sharp by taking $g=g_{\bR^n}$. If we relax the condition that $g|_{\p\Omega}=g_{\bR^n}|_{\p\Omega}$, a non-sharp estimate was later obtained by Shi, Wang, Wei, and Zhu:
\begin{theorem}[Shi-Wang-Wei-Zhu \cite{SWWZ}]{\label{thm:SWWZ}}
    For $n\ge 3$, let $\sigma$ be a Riemannian metric on $S^{n-1}$ that is isotopic to the standard round metric through a path of metrics with positive scalar curvature. Then there exists some constant $C=C(n,\sigma)$ with the following significance:
    
    If $g$ be a Riemannian metric on $B^n$ and let $H_g$ denotes the mean curvature of $\p B^n=S^{n-1}$ with respect to the metric $g$. Suppose $R_g\ge 0$ in $B^n$, $g|_{\p B^n}=\sigma$ on $\p B^n$ and  $H_g>0$ at each point on $S^{n-1}$, then
    \[
    \int_{S^{n-1}}H_gd\vol_{\sigma}\le C.
    \]    
\end{theorem}
Roughly speaking, the lower bound of the scalar curvature in the interior should control the mean curvature of the boundary from above. These results were later generalized by the work of Shi-Wang-Wei under the condition where the scalar curvature is bounded from below by $-n(n-1)$.
\begin{theorem}[Shi-Wang-Wei \cite{SWW}]{\label{thm:sharp-AH-estimate}}
    Let $g$ be a Riemannian metric defined on $B^n$ with
    \begin{itemize}
        \item $R_g\ge -n(n-1)$ in $B^n$.
        \item $H_g>0$ on $\p B^n$.
        \item $g|_{\p B^n}=\lambda^2 g_{S^{n-1}}$ for some constant $\lambda>0$, and $g_{S^{n-1}}$ denotes the standard spherical metric on $S^{n-1}$. 
    \end{itemize}
    Then
        \[
        \int_{\p B^n}H_gd\sigma_g\le (n-1)\omega_{n-1}\sqrt{1+\lambda^2}.
        \]
\end{theorem}
We also have the analogue of Theorem \ref{thm:SWWZ} in the hyperbolic setting. 
\begin{theorem}[Shi-Wang-Wei \cite{SWW}]{\label{thm:integral-H-AF}}
    For $n\ge 3$, let $\sigma$ be a Riemannian metric on $S^{n-1}$. Then there exists some constant $C=C(n,\sigma)$ with the following significance:
    
    If $g$ be a Riemannian metric on $B^n$ and let $H_g$ denotes the mean curvature of $\p B^n=S^{n-1}$ with respect to the metric $g$. Suppose $R_g\ge -n(n-1)$ in $B^n$, $g|_{\p B^n}=\sigma$ on $\p B^n$ and  $H_g>0$ at each point on $S^{n-1}$, then
    \[
    \int_{S^{n-1}}H_gd\vol_{\sigma}\le C.
    \]    
\end{theorem}
These results are closely related to the following conjecture by Gromov \cite{gromov2019scalarcurvaturemanifoldsboundaries}, \cite{four-lectures}, where several progresses have been made in \cite{Mantoulidis-Miao}, \cite{Mantoulidis-Miao-Tam}.
\begin{conjecture}
    Let $\kappa\in \bR$ be a constant, and let $(M^n,g)$ be a compact Riemannian manifold with boundary. Then there exists some constant $C=C(n,\kappa,\p M, g|_{\p M})$ such that if $R_g\ge \kappa$, then
    \[
    \int_{\p M}H\le C.
    \]
    where $H$ is the mean curvature of the boundary with respect to the outward normal.
\end{conjecture}
A key ingredient in the proof of the above results is the positive mass theorem: Given a fill-in of the round $S^{n-1}$ with non-negative scalar curvature and positive mean curvature on the boundary, one can extend the metric outward by solving a non-linear parabolic equation. It turns out that one obtains an asymptotically flat (resp. asymptotically hyperbolic) manifold $X^n$ with $R\ge 0$ (resp. $R\ge -n(n-1)$). Next, by using a monotonicity formula along the parabolic equation, one can estimate $\int_{S^{n-1}}H$ in terms of the mass of $X$: if the integral of the mean curvature is too large, then the mass of $X$ has to be negative, contradicting the positive mass theorem. 

Indeed, the estimate of the integral of the mean curvature has a natural interpretation as a notion of quasi-local mass in mathematical relativity. We refer to the work of Brown-York \cite{Brown-York-1}, \cite{Brown-York-2}, Shi-Tam \cite{Shi-Tam}, Liu-Yau \cite{Liu-Yau}, \cite{Liu-Yau-2} and Wang-Yau \cite{Wang-Yau} for this direction.

The classical positive mass theorem was first proved by Schoen and Yau \cite{SY-PMT} and Witten \cite{Witten-PMT}. The positive mass theorem in the asymptotically hyperbolic setting is discussed in \cite{ACG-PMT}, \cite{CH-PMT}, and \cite{Wang-PMT}. Recently, Brendle and Hung \cite{Brendle-Hung} deduced a novel positive mass theorem in the asymptotically locally hyperbolic setting, resolving a long-standing conjecture by Horowitz and Myers \cite{HM}. In particular, the conformal infinity of the model is a torus. Motivated by the principle relating the total mean curvature estimate and the positive mass theorem above, it is natural to consider the analogue of the Shi-Tam result in this setting.

In this note we prove the following results:
\begin{theorem}{\label{thm:sharp-total-mean-curvature}}
 Suppose $3\le n\le 7$. Let $\gamma$ be a flat metric on $\Sigma=S^1\times T^{n-2}$. Let $\Omega=B^2\times T^{n-2}$ equipped with an arbitrary Riemannian metric $g$ such that $R_g\ge -n(n-1)$ and $(\p\Omega,g|_{\p\Omega})=(\Sigma,\gamma)$. Let $H_{\Sigma}$ denote the mean curvature of $\Sigma$ with respect to the $g$ outward unit normal vector. Let $\Xi$ denotes the pull-back of the volume form on $S^1$ under the canonical projection from $\Sigma$ to $S^1$. If $H_{\Sigma}>0$, then
 \[
 \frac{1}{\vol(\Sigma)}\int_{\Sigma}(H_{\Sigma}-(n-1))d\vol_{\gamma}\le\frac{1}{2}\left(\frac{4\pi}{n\sigma}\right)^n,
 \]
 where $\sigma$ denotes the length of the shortest closed curve $\alpha$ in $(\Sigma,\gamma)$ satisfying $\int_{\alpha}\Xi\ne 0$.
\end{theorem}
The estimate in Theorem \ref{thm:sharp-total-mean-curvature} is sharp. One can see this by taking $\Omega_i=\{r\le i\}\subset \bR^2\times T^{n-2}$ with the Horowitz-Myers metric on $\bR^2\times T^{n-2}$ and letting $i\to\infty$, where $r$ denotes the radial coordinate in the $\bR^2$ factor. We refer to Remark 1.7 in \cite{Brendle-Hung} for further discussions. 

We also prove an interpolation result based on the techniques developed in \cite{SWW}.
\begin{theorem}{\label{thm:interpolation}}
    Let $\hat{\gamma},\gamma$ be Riemannian metrics on $\Sigma=S^1\times T^{n-2}$ with $\gamma>\hat{\gamma}$ in the sense of symmetric $(0,2)$-tensor. Suppose $\hat{g}$ is a Riemannian metric on $\Omega=B^2\times T^{n-2}$ with 
    \begin{itemize}
        \item $({\p\Omega},\hat{g}|_{\p\Omega})=(\Sigma,\hat{\gamma})$.
        \item $R_{\hat{g}}\ge -n(n-1)$ in $\Omega$.
        \item $H_{\hat{g}}>0$ on $\p\Omega$.
    \end{itemize}  
    Then, there exists some Riemannian metric $g$ on $\Omega$ with
    \begin{itemize}
        \item $(\p\Omega,g|_{\p\Omega})=(\Sigma,\gamma)$.
         \item $R_{g}\ge -n(n-1)$ in $\Omega$.
        \item  $H_{g}>0$ on $\p\Omega$. 
    \end{itemize}
    Moreover, there exists some constant $C=C(\Sigma, \gamma,\hat{\gamma},n)$ with
        \[
        \int_{\p\Omega}H_{\hat{g}}d\vol_{\hat{\gamma}}\le C \int_{\p\Omega}H_{g}d\vol_{{\gamma}}.
        \]
\end{theorem}
As a consequence of Theorem \ref{thm:sharp-total-mean-curvature} and \ref{thm:interpolation}, we are able to deduce the following integral estimates of mean curvature, analogous to Theorem \ref{thm:integral-H-AF}.
\begin{theorem}
    Suppose $3\le n\le 7$. Let $\hat{\gamma}$ be a Riemannian metric on $\Sigma=S^1\times T^{n-2}$. There exists some constant $C=C(\hat{\gamma},n)$ with the following significance:

    Suppose $\hat{g}$ is a Riemannian metric on $\Omega =B^2\times T^{n-2}$ with $\hat{g}|_{\p\Omega}=\hat{\gamma}$ and $R_{\hat{g}}\ge -n(n-1)$ in $\Omega$. Let $H_{\hat{g}}$ be the mean curvature of $\p\Omega$ with respect to the Riemannian metric $\hat{g}$. Assume $H_{\hat{g}}>0$ on $\p\Omega$, then
    \[
    \int_{\p\Omega}H_{\hat{g}}d\vol_{\hat{\gamma}}\le C.
    \]
\end{theorem}
\begin{proof}
    We pick an arbitrary flat metric $\gamma$ on $\Sigma$ with $\gamma>\hat{\gamma}$. By Theorem \ref{thm:interpolation}, there exists some Riemannian metric $g$ on $\Omega$ with $g|_{\p\Omega}=\gamma$, $H_g>0$ and $R_g\ge -n(n-1)$ in $\Omega$. Moreover, there exists some constant $C=C(n,\gamma,\gamma')$ with the property that
    \[
    \int_{\p\Omega}H_{\hat{g}}d\vol_{\hat{\gamma}}\le C\int_{\p\Omega}H_{g}d\vol_{\gamma}.
    \]
    On the other hand, we may apply Theorem \ref{thm:sharp-total-mean-curvature} to $(\Omega,g)$ to conclude that
    \[
    \int_{\p\Omega}H_gd\vol_{\gamma}\le C(n,\gamma).
    \]
    As $\gamma$ is chosen arbitrarily, the result follows directly.
\end{proof}
\subsection*{Acknowledgement} The author wishes to express sincere gratitude to his advisor, Simon Brendle, for inspiring discussions and continued support. Additionally, the author would like to thank Bernhard Hanke and Christian Bär for suggesting the problem and for valuable discussions.
\section{Proof of Theorem \ref{thm:sharp-total-mean-curvature}}
\subsection{The Prescribed Scalar Curvature Equation}
Let $\Sigma^{n-1}=S^1\times T^{n-2}$ be the $(n-1)$ dimensional torus with some flat Riemannian metric $\gamma$. We endow $\Sigma\times [\rho_0,\infty)$ with the hyperbolic metric $\overline{g}= \rho^2 \gamma+\rho^{-2}d\rho\otimes d\rho$. We will denote $\Sigma_{\rho}=\Sigma\times\{\rho\} $ to be the constant $\rho$ slice. The scalar curvature of $\overline{g}$ is $-n(n-1)$. We denote $\overline{A}_{\Sigma_{\rho}}$ and $\overline{H}_{\Sigma_{\rho}}$ to be the second fundamental form  and the mean curvature of $\Sigma_{\rho}$ with respect to the $\overline{g}$-unit normal vector $\rho\frac{\p}{\p \rho}$. In particular, $\overline{A}_{\Sigma_{\rho}}=\rho^2\gamma$ and $\overline{H}_{\Sigma_{\rho}}=n-1$.

Let $u(x,\rho)$ be a smooth positive function on $\Sigma\times [\rho_0,\infty)$. We consider the Riemannian metric $g=\rho^2\gamma+u^2\rho^{-2}d\rho\otimes d\rho$.  We let $A_{\Sigma_{\rho}},H_{\Sigma_{\rho}}$ to be the second fundamental form and the mean curvature with respect to $g$ and the unit normal $\nu_{g}=u^{-1}\rho\frac{\p}{\p \rho}$. Clearly, we have $A_{\Sigma_{\rho}}=u^{-1}\rho^2\gamma$ and $H_{\Sigma_{\rho}}=(n-1)u^{-1}$ with respect to the metric $g$.
\begin{lemma}{\label{lem:prescribed-scalar}}
    The condition $R_g=-n(n-1)$ is equivalent to the following equation:
    \begin{equation}{\label{eqn:prescribed-scalar}}
    (n-1)\rho\frac{\p u}{\p \rho}=u^2\lp_{\Sigma_{\rho}}u-\frac{1}{2}n(n-1)(u^3-u).
    \end{equation}
\end{lemma}
\begin{proof}
    We consider the normal variation of the slice $\Sigma_{\rho}$ along the vector field $\rho \frac{\p}{\p \rho}=u\nu_{g}$, then the variation formula of mean curvature and the Gauss equation implies
    \begin{align*}
        &(n-1)\rho \frac{\p}{\p \rho}u^{-1}=\rho \frac{\p}{\p \rho}{H}\\
        &=-\lp_{\Sigma_{\rho}}u+\frac{1}{2}(R_{\Sigma_{\rho}}-R_{{g}}-|{A}_{\Sigma_{\rho}}|_g^2-{H}_{\Sigma_{\rho}}^2)u\\
        &=-\lp_{\Sigma_{\rho}}u-\frac{1}{2}R_{{g}}u-\frac{1}{2}n(n-1)u^{-1}.
    \end{align*}
    Prescribing $R_g=-n(n-1)$ is then equivalent to
    \begin{equation}{\label{eqn:u-inverse}}
        (n-1)\rho \frac{\p}{\p \rho}u^{-1}=-\lp_{\Sigma_{\rho}}u+\frac{1}{2}n(n-1)(u-u^{-1}).
    \end{equation}
    The assertion \eqref{eqn:prescribed-scalar} then follows from $\frac{\p }{\p \rho}u^{-1}=-u^{-2}\frac{\p u}{\p \rho}$.
\end{proof}
The prescribed scalar curvature equation \eqref{eqn:prescribed-scalar} was discussed by Bartnik \cite{Bartnik}. We refer to \cite{SWW} for a general analysis of this type of equation.
\begin{lemma}{\label{lem:monotonicity}}
    Suppose $u$ is a positive solution of \eqref{eqn:prescribed-scalar}. Then we have the following monotonicity formula
    \begin{equation}{\label{eqn:monotonicity}}
        \rho \frac{d}{d\rho}\int_{\Sigma_{\rho}}((n-1)-H_{\Sigma_{\rho}})\rho d\vol_{\Sigma_{\rho}}=-\frac{n(n-1)}{2}\int_{\Sigma\rho}u^{-1}(1-u)^2\rho d\vol_{\Sigma_{\rho}}\le 0.
    \end{equation}
\end{lemma}
\begin{proof}
    Using \eqref{eqn:u-inverse}, we can calculate
    \begin{align*}
        &\rho \frac{\p}{\p \rho}\left(\rho^n((n-1)-H_{\Sigma_{\rho}})\right)\\
        &=n(n-1)\rho^n(1-u^{-1})+\rho^n(\lp_{\Sigma_{\rho}}u-\frac{1}{2}n(n-1)(u-u^{-1}))\\
        &=\frac{1}{2}n(n-1)\rho^n(2-u-u^{-1})+\rho^{n-2}\lp_{\Sigma}u.
    \end{align*}
    Integrating over $\Sigma$ and applying the divergence Theorem, we obtain
    \begin{align*}
        & \rho \frac{d}{d\rho}\int_{\Sigma_{\rho}}((n-1)-H_{\Sigma_{\rho}})\rho d\vol_{\Sigma_{\rho}}\\
        &=\int_{\Sigma}\rho \frac{\p}{\p \rho}\left(\rho^n((n-1)-H_{\Sigma_{\rho}})\right)d\vol_{\gamma}\\
        &=\frac{n(n-1)}{2}\rho\int_{\Sigma_{\rho}}(2-u^{-1}-u)d\vol_{\Sigma_{\rho}}.
    \end{align*}
    The claim now follows directly.
\end{proof}
\subsection{A Priori Estimates}
In this section, we suppose $u$ is a solution to \eqref{eqn:prescribed-scalar} with initial data $u_0>0$. In other words, we assume $u$ is a solution of the following equation
\begin{equation}{\label{eqn:PDE-warped-product}}
\begin{cases}\rho\frac{\p u}{\p \rho}=\frac{u^2}{(n-1)}\lp_{\Sigma_{\rho}}u +\frac{n}{2}(u-u^3)\\
u(x,\rho_0)=u_0(x)
\end{cases}
\end{equation}
for some positive function $u_0(x)$ defined on $\Sigma$. By the short-time existence of parabolic equation, there exists some $R>\rho_0$ such that the solution $u(x,\rho)$ exists on $\Sigma\times [\rho_0,R)$.
\begin{lemma}{\label{lem:C0}}
    There exists some constant $C_0$ that depends only on $u_0$ and $n$, such that for $\rho_0 \le \rho <R$, we have
    \begin{equation}{\label{eqn:barrier}}
    \left(1+C_0\rho^{-n}\right)^{-\frac{1}{2}}\le u\le \left(1-C_0\rho^{-n}\right)^{-\frac{1}{2}}.
    \end{equation}
    In particular, if $u$ is defined on $[\rho_0,\infty)$ there exists some constant $C=C(n,\Sigma)$ with
    \[
    |u(x,\rho)-1|\le C\rho^{-n}
    \]
    for all sufficiently large $\rho$.
\end{lemma}
\begin{proof}
We study solutions of the corresponding ODE
\[
    \rho\frac{d \varphi}{d\rho}=\frac{n}{2}(\varphi-\varphi^3),
\]
which can be solved explicitly over $[\rho_0,\infty)$ by
\[
    \varphi=(1+C\rho^{-n})^{-\frac{1}{2}},
\]
where $C$ depends only the initial data $\varphi(\rho_0)$. The claim now follows from the standard maximum principle.
\end{proof}
\begin{lemma}[\cite{Shi-Tam}, \cite{SWW}, \cite{Wang-Yau}]{\label{lem:long-time}}
     Suppose $u(x,\rho)$ solves \eqref{eqn:PDE-warped-product} on $\Sigma\times [\rho_0,\infty)$. Then for each $m,k\ge 0$, there exists some constant $C=C(n,\Sigma,m,k)$ such that
    \begin{equation}{\label{eqn:derivative-estimate}}
    \left|\left(\rho\frac{\p}{\p \rho}\right)^m (\rho^{-1}D_{\Sigma})^{k}(u-1)\right|\le C \rho^{-n}
    \end{equation}
    for all sufficiently large $\rho$.
\end{lemma}
\begin{proof}
We set $v=u-1$ and introduce the parameter $t=\log \rho$. Then equation \eqref{eqn:PDE-warped-product} can be rewritten as
\begin{equation}{\label{eqn:rescaled-equation}}
\frac{\p v}{\p t}=\frac{(v+1)^2}{n-1}\lp_{\Sigma_{\rho}}v-\frac{n}{2}(v+2)(v+1)v.
\end{equation}
By Lemma \ref{lem:C0}, for $t$ (hence $\rho$) large enough, there exists a universal constant $\Lambda=\Lambda(n,\Sigma)$ such that
\[
0<\Lambda^{-1}\le \frac{(v+1)^2}{n-1}\le \Lambda,\qquad \left|\frac{n}{2}(v+2)(v+1)v\right|\le \Lambda
\]
for all $x\in\Sigma$ and all sufficiently large $\rho$.

Fix some $x_0\in \Sigma$ and $\bar{\rho}\ge 1$ large, we define the parabolic cylinder for $r>0$ by
\begin{align*}
&Q_r:=B^{\Sigma}_{\frac{r}{\bar{\rho}}}(x_0)\times (e^{-r^2}\bar{\rho},\bar{\rho}]\\
&\simeq \left\{(x,t)\in \Sigma\times (\bar{t}-r^2,\bar{t}]:d_{\bar{g}}\left((x,t),(x_0,t)\right)<re^{t-\bar{t}}\right\},
\end{align*}
where $B_{r}^{\Sigma}(x_0)$ denotes the geodesic ball of radius $r$ within $(\Sigma,\gamma)$ centered at $x_0$ and $\bar{t}=\log\bar{\rho}$. 

We consider equation \eqref{eqn:PDE-warped-product} within the parabolic cylinder $Q_1$. It follows from the parabolic Krylov-Safanov estimate (see also \cite{LSU}) that there exists some constant $\alpha<1$ and $C$ independent of $x_0$ and $\bar{\rho}$ such that
\[\|v\|_{C^{\alpha,\frac{\alpha}{2}}(Q_{\frac{1}{2}})}\le C(n,\Sigma)\|v\|_{C^0(Q)}.
\]
Next, we apply the parabolic interior Schauder estimate \cite{LSU} and a bootstrap argument to conclude that
\[
    \sup_{Q_{\frac{1}{4}}}\left|\left(\frac{\p}{\p t}\right)^m D^k_{\Sigma_{\rho}}v\right|\le C(m,k,n,\Sigma) \|v\|_{C^0(Q)}. 
\]
In the original coordinates, this is equivalent to
\[
    \sup_{B^{\Sigma}_{\frac{1}{4\bar{\rho}}}(x_0)\times (\exp(-1\slash 16){\bar{\rho}},\bar{\rho}]}\left|\left(\rho\frac{\p}{\p \rho}\right)^m (\rho^{-1}D_{\Sigma})^{k}v\right|\le C(m,k,n,\Sigma) \|v\|_{C^0(Q)}. 
\]
The claim now follows from Lemma \ref{lem:C0}.
\end{proof}
\begin{corollary}{\label{cor:long-time-existence}}
    There exists a positive solution $u(x,\rho)$ of \eqref{eqn:PDE-warped-product}
    on $\Sigma\times [\rho_0,\infty)$ satisfying the estimate \eqref{eqn:derivative-estimate} for all sufficiently large $\rho$.
\end{corollary}
We now define $
\psi(x,\rho)=\rho^n(u-1)
$.
Then it follows directly from Lemma \ref{lem:long-time} that
\[
    \left|\left(\rho\frac{\p}{\p \rho}\right)^m (\rho^{-1}D_{\Sigma})^{k}\psi\right|\le C(m,k,n,\Sigma).
\]
\begin{lemma}[\cite{Shi-Tam}, \cite{SWW}, \cite{Wang-Yau}]{\label{lem:mass-estimate}}
    For each $m,k\ge 0$, there exists some constant $C=C(n,\Sigma,m,k)$ so that
    \[
   \left|\left(\rho^3\frac{\p}{\p \rho}\right)^m D_{\Sigma}^k\psi\right|\le C
    \]
    for all sufficiently large $\rho$.
\end{lemma}
\begin{proof}
    It follows directly from equation \eqref{eqn:PDE-warped-product} that $\psi$ satisfies the equation
    \begin{equation}{\label{eqn:mass-equation}}
        \rho^3\frac{\p \psi}{\p \rho}=\frac{u^2}{n-1}\lp_{\Sigma} \psi+(1-\frac{u+u^2}{2}){n}\rho^2\psi.
    \end{equation}
    The estimate \eqref{eqn:derivative-estimate} implies that we can write \eqref{eqn:mass-equation} as
    \[
    \rho^3\frac{\p\psi}{\p\rho}=\frac{1}{n-1}\lp_{\Sigma}\psi+O(\rho^{2-n})
    \]
    for all sufficiently large $\rho$. Define $s=1-2\rho^{-2}$, then for $s\in [1-2\rho_0^{-2},1)$ we have
    \[
    \frac{\p\psi}{\p s}=\frac{1}{n-1}\lp_{\Sigma}\psi +O\left((1-s)^{\frac{n}{2}-1}\right).
    \]
    Standard parabolic estimates then imply that 
    \[
    \sup_{(x,s)\in\Sigma\times [1-2\rho_0^{-2},1)}\left|\left(\frac{\p}{\p s}\right)^m D_{\Sigma}^k\psi\right|\le C(m,k,n,\Sigma).
    \]
    This is equivalent to stating that
    \[
    \sup_{(x,\rho)\in\Sigma[1,\infty)}\left|\left(\rho^3\frac{\p}{\p \rho}\right)^m D_{\Sigma}^k\psi\right|\le C(m,k,n,\Sigma).
    \]
    This proves the claim.
\end{proof}
\begin{corollary}{\label{cor:mass-expansion}}
    There exists a smooth function $\mu(x)$ defined on $\Sigma$ such that, for $\rho$ sufficiently large, there exists a constant $C=C(n,\Sigma)$ with
    \[
    \sup_{x\in \Sigma}\left|u(x,\rho)-1-{\mu(x)}{\rho^{-n}}\right|\le C{\rho^{-n-2}}.
    \]
    In particular, this implies 
\[
\left|H_{\Sigma_{\rho}}-(n-1)(1-\mu(x){\rho^{-n}})\right|\le C{\rho^{-n-2}}
\]
over $\Sigma_{\rho}$.
\end{corollary}
\begin{proof}
It follows from Lemma \ref{lem:mass-estimate} that $\psi(x,\rho)$ converges smoothly to some smooth function $\mu(x)$ defined on $\Sigma$. Moreover, the derivative estimates imply that there exists a constant $C(n,\Sigma)$ such that, for $\rho$ sufficiently large, we have $\sup_{\Sigma}|\psi(x,\rho)-\mu(x)|\le C\rho^{-2}$. The conclusion follows from the definition of $\psi(x,\rho)$. The estimate for $H_{\Sigma_{\rho}}$ follows from the relation $H_{\Sigma_{\rho}}=(n-1)u^{-1}$.
\end{proof}
We let $\mu(x)$ be defined as in Corollary \ref{cor:mass-expansion}, and we set
\begin{equation}{\label{eqn:average-mass-expansion}}
\mu_0=\frac{1}{\vol(\Sigma)}\int_{\Sigma}\mu(x)d\vol_{\gamma}
\end{equation}
to be the average of $\mu$ on $\Sigma$. We let $f$ be the unique solution of the PDE
\begin{equation}{\label{eqn:harmonic-perturbation}}
    \lp_{\Sigma}f=(n-1)(\mu_0-\mu),\qquad \int_{\Sigma}f(x)d\vol_{\gamma}=0.
\end{equation}
For each $\lambda\ge 1$, we define $\hat{M}_{\lambda}=M\backslash \{\rho>\lambda+\lambda^{3-n}f\}$ with $\p\hat{M}_{\lambda}=\hat{\Sigma}_{\lambda}$. The following perturbation argument is similar to the one used in \cite{Brendle-Hung}.
\begin{proposition}{\label{prop:almost-CMC}}
    The mean curvature of $\hat{\Sigma}_{\lambda}$ with respect to $g$ satisfies
    \begin{equation}{\label{eqn:mean-curvature-perturbation}}
    H_{\hat{\Sigma}_{\lambda}}=(n-1)(1-\lambda^{-n}\mu_0)+o(\lambda^{-n})
    \end{equation}
    for all sufficiently large $\lambda$.
\end{proposition}
\begin{proof}
    For $\varphi:M\to \bR$ and if $0$ is a regular value of $\varphi$, the mean curvature of the level set $\{\varphi=0\}$ is given by
    \begin{equation}{\label{eqn:level-set-H-formula}}
    \frac{1}{|\gd \varphi|_g}\left(\lp \varphi-D^2\varphi(\frac{\gd \varphi}{|\gd \varphi|_g},\frac{\gd \varphi}{|\gd \varphi|_g})\right)
    \end{equation}
    where the covariant derivative and the Laplacian are computed with respect to the metric $g$. Let $V=u^{-2}\rho^2\frac{\p}{\p\rho}$ denote the gradient of the function $\rho$ with respect to $g$. We can compute
    \begin{equation}{\label{eqn:hessian-rho}}
    \begin{split}
        &D^2\rho=\frac{1}{2}\mathscr{L}_Vg\\
    &=\left(\rho^{-1}-u^{-1}\frac{\p u}{\p \rho}\right)d\rho\otimes d\rho-u^{-1}u_i(d\rho\otimes dx^i+dx^i\otimes d\rho)+u^{-2}\rho^3\gamma.
    \end{split}
     \end{equation}
    Next, we extend the function $f(x)$ trivially to $M$ and let $W^i=\rho^{-2}f_i$ denote the gradient of $f$ with respect to $g$. Then
     \begin{equation}{\label{eqn:hessian-f}}
    \begin{split}
    &D^2f=\frac{1}{2}\mathscr{L}_Wg\\
    &=\rho^{-4} g(\gd^{\Sigma}f,\gd^{\Sigma}u) d\rho\otimes d\rho-\rho^{-3}f_i(d\rho\otimes dx^i+dx^i\otimes d\rho)+f_{ij}dx^i\otimes dx^j.
    \end{split}
     \end{equation}
    Define $\varphi=\rho-\lambda-\lambda^{3-n}f$ and $\hat{\Sigma}_{\lambda}=\{\varphi=0\}$, then the gradient of $\varphi$ with respect to $g$ is given by
    \begin{align*}
    \gd\varphi=u^{-2}\rho^2\frac{\p}{\p \rho}-\lambda^{3-n}\rho^{-2}f_i\frac{\p}{\p x_i},
    \end{align*}
    It follows that $$|\gd\varphi|_g^2=u^{-2}\rho^2+\lambda^{6-2n}\rho^{-2}|\gd^{\Sigma} f|^2.$$
    Therefore, by the derivative estimate \eqref{eqn:derivative-estimate}, it holds on $\hat{\Sigma}_{\lambda}$ that
    \begin{equation}{\label{eqn:gradient-norm}}
        \frac{1}{|\gd\varphi|_g}=u(\lambda+\lambda^{3-n}f)^{-1}+o(\lambda^{-1-n}),
    \end{equation}
    and
    \begin{equation}{\label{eqn:gradient-normal}}
    \left|\frac{\gd\varphi}{|\gd\varphi|_g}-u^{-1}(\lambda+\lambda^{3-n}f)\frac{\p}{\p\rho}+\lambda^{3-n}(\lambda+\lambda^{3-n}f)^{-3}f_i\frac{\p}{\p x_i}\right|_{\bar{g}}\le o(\lambda^{-n}).
    \end{equation}
   where we recall that $\bar{g}$ is the standard hyperbolic metric. Using \eqref{eqn:gradient-normal}, we compute the Hessian:
    \begin{align*}
        &(D^2\varphi)\left(\frac{\gd\varphi}{|\gd\varphi|_g},\frac{\gd\varphi}{|\gd\varphi|_g}\right)\\
        =&u^{-2}(\lambda+\lambda^{3-n}f)^2(D^2\varphi)\left(\frac{\p}{\p \rho},\frac{\p}{\p \rho}\right)\\
        &+\lambda^{6-2n}(\lambda+\lambda^{3-n}f)^{-6}f_if_j(D^2\varphi)\left(\frac{\p}{\p x_i},\frac{\p}{\p x_j}\right)\\
        &-2\lambda^{3-n}(\lambda+\lambda^{3-n}f)^{-2}u^{-1}f_i(D^2\varphi)\left(\frac{\p}{\p \rho},\frac{\p}{\p x_i}\right)\\
        =&u^{-2}(\lambda+\lambda^{3-n}f)^2(D^2\varphi)\left(\frac{\p}{\p \rho},\frac{\p}{\p \rho}\right)\\
        &-2\lambda^{3-n}(\lambda+\lambda^{3-n}f)^{-2}u^{-1}f_i(D^2\varphi)\left(\frac{\p}{\p \rho},\frac{\p}{\p x_i}\right)+o(\lambda^{1-2n}).
    \end{align*}
    Therefore, we can compute
    \begin{align*}
        &\lp\varphi-(D^2\varphi)\left(\frac{\gd\varphi}{|\gd\varphi|_g},\frac{\gd\varphi}{|\gd\varphi|_g}\right)\\
        &=\sum_{i=1}^{n-1}(\lambda+\lambda^{3-n}f)^{-2}(D^2\varphi)\left(\frac{\p}{\p x_i},\frac{\p}{\p x_i}\right)\\
        &+2\lambda^{3-n}(\lambda+\lambda^{3-n}f)^{-2}u^{-1}f_i(D^2\varphi)\left(\frac{\p}{\p \rho},\frac{\p}{\p x_i}\right)+o(\lambda^{-n}).
    \end{align*}
    We use \eqref{eqn:hessian-rho} and \eqref{eqn:hessian-f} to see that
    \[
    u^{-1}(D^2\varphi)\left(\frac{\p}{\p \rho},\frac{\p}{\p x_i}\right)=o(\lambda^{-2}),
    \]
    and
    \begin{align*}
        &\sum_{i=1}^{n-1}(\lambda+\lambda^{3-n}f)^{-2}(D^2\varphi)\left(\frac{\p}{\p x_i},\frac{\p}{\p x_i}\right)\\
        &=(n-1)u^{-2}(\lambda+\lambda^{3-n}f)-\lambda^{3-n}(\lambda+\lambda^{3-n}f)^{-2}\lp_{\Sigma}f.
    \end{align*}
    Together with \eqref{eqn:gradient-norm}, we conclude that
    \begin{align*}
        &\frac{1}{|\gd\varphi|_g}\left(\lp\varphi-(D^2\varphi)\left(\frac{\gd\varphi}{|\gd\varphi|_g},\frac{\gd\varphi}{|\gd\varphi|_g}\right)\right)\\
        &=(n-1)u^{-1}-\lambda^{3-n}(\lambda+\lambda^{3-n}f)^{-3}u^{-1}\lp_{\Sigma}f\\
        &=(n-1)(1-\mu(x)\lambda^{-n})-\lambda^{-n}\lp_{\Sigma}f+o(\lambda^{-n}),
    \end{align*}
    where in the last line we used the expansion of $u$ by Corollary \ref{cor:mass-expansion}. Since $f$ solves the equation \eqref{eqn:harmonic-perturbation}, the result follows directly.
\end{proof}
We are now ready to prove Theorem \ref{thm:sharp-total-mean-curvature}, closely following the strategy in \cite{Brendle-Hung}.
\begin{proof}[Proof of Theorem \ref{thm:sharp-total-mean-curvature}]
Given $\ve>0$ to be determined later. By Corollary \ref{cor:long-time-existence}, there exists a smooth function $u_{\ve}(x,\rho)$ defined on $\Sigma\times [1,\infty)$ that solves equation \eqref{eqn:PDE-warped-product} with the initial condition 
$$u_{\ve}(x,1)=(1-\ve)^{-1}\frac{n-1}{H_{\Sigma}}.$$
We let $g_{u_{\ve}}:=\rho^2\gamma+{u_{\ve}}^2\rho^{-2}d\rho\otimes d\rho$ to be a Riemannian metric on $\Sigma\times [1,\infty)$.

By Corollary \ref{cor:mass-expansion}, there exists some smooth function $\mu^{\ve}(x)$ define on $\Sigma$, such that for $\rho$ large enough, it holds that
\[
\sup_{x\in \Sigma}|u_{\ve}-1-\mu^{\ve}\rho^{-n}|\le C\rho^{-n-2}
\]
for some constant $C$ independent of $\ve$. We set $\mu^{\ve}_0$ to be the average of $\mu^{\ve}$ on $\Sigma$ as \eqref{eqn:average-mass-expansion}.

For $\rho\ge 1$, we consider the quantity
\[
\cF(\rho)=\frac{\rho}{\on{vol}(\Sigma)}\int_{\Sigma_{\rho}}((n-1)-H_{\Sigma_{\rho}})d\vol_{\Sigma_{\rho}},
\]
which is monotone decreasing in $\rho$ by \eqref{eqn:monotonicity}. In particular,
\begin{align*}
\cF(1)&=\frac{1}{\on{vol}(\Sigma)}\int_{\Sigma}((n-1)-(1-\ve)H_{\Sigma})d\vol_{\Sigma}\\
&\ge \lim_{\rho\to\infty}\cF(\rho)\\
&=(n-1)\mu_0^{\ve},
\end{align*}
where the last inequality is due to Corollary \ref{cor:mass-expansion}. 

Suppose the claim is not true and
\[
\frac{1}{\on{vol}(\Sigma)}\int_{\Sigma}\left((n-1)-H_{\Sigma}\right)d\vol_{\Sigma}<-\frac{1}{2}\left(\frac{4\pi}{n\sigma}\right)^n.
\]
We will then choose $\ve=\ve(n,\Sigma)$ small so that
\[
(n-1)\mu_0^{\ve}\le \cF(1)\le -\frac{1}{2}(1-\ve)^{-n-1}\left(\frac{4\pi}{n\sigma}\right)^n.
\]
In particular, we have $\mu_0^{\ve}<0$. 

We can extend $\Omega$ to a smooth Riemannian metric, on $M^n:=\bR^2\times T^{n-2}$ as follows: Note that $g|_{\p\Omega}=g_{u_{\ve}}|_{\Sigma\times\{1\} }$, and the choice of $u_{\ve}(x,1)$ guarantees that $H_{\Sigma_1}= (1-\ve)H_{\Sigma}$. Lemma \ref{lem:gluing} implies that for any $\delta>0$ to be specified later, we can glue $(\Omega,g)$ with $(\Sigma\times [1,\infty),g_{u_{\ve}})$ together along the boundary, yielding a metric $g_0$ on $M$ with the following properties:
\begin{itemize}
    \item $R_{g_0}\ge -(1-\delta)^{-2}n(n-1)$.
    \item $g_0$ agrees with $g_{u_{\ve}}$ on $\Sigma\times [2,\infty)$.
\end{itemize}
Let $f_{\ve}$ to be the unique solution of the equation
\[
\lp_{\Sigma}f_{\ve}=(n-1)(\mu^{\ve}_0-\mu^{\ve}),\qquad \int_{\Sigma}f_{\ve} d\vol_{\gamma}=0.
\]
We then define $$\hat{M}_{\lambda}=M\backslash \{\rho>\lambda+\lambda^{3-n}f_{\ve}\},$$
which makes sense under the metric $g_0$ for $\lambda$ large enough. By Proposition \ref{prop:almost-CMC} we can choose $\lambda=\lambda(\ve,\Sigma)$ large enough such that
\begin{align*}
H_{\hat{\Sigma}_{\lambda}}&\ge (n-1)-(n-1)(1-\ve)\lambda^{-n}\mu_0^{\ve}\\
&\ge (n-1)+\frac{1}{2}(1-\ve)^{-n}\lambda^{-n}\left(\frac{4\pi}{n\sigma}\right)^n
\end{align*}
over $\hat{\Sigma}_{\lambda}$. 

We may rescale the metric $g_0$ by some $\tilde{g}_0=(1-\delta)^{-2}g_0$ so that $R_{\tilde{g}}\ge -n(n-1)$. Let $\tilde{H}_{\hat{\Sigma}_{\lambda}}=(1-\delta)H_{\hat{\Sigma}_{\lambda}}$ denotes the mean curvature of $\hat{\Sigma}_{\lambda}$ with respect to the metric $\tilde{g_0}$. We will then choose $\delta=\delta(\ve,\lambda,\Sigma)$ so that
\[
\tilde{H}_{\hat{\Sigma}_{\lambda}}-(n-1)\ge \frac{1}{2}\left(1-\frac{\ve}{2}\right)^{-n}\lambda^{-n}\left(\frac{4\pi}{n\sigma}\right)^n.
\]
Moreover, by choosing $\delta$ smaller, we can assume that
\[
\tilde{\sigma}\ge \left(1-\frac{\ve}{4}\right)\hat{\sigma},
\]
where $\tilde{\sigma}$ and $\hat{\sigma}$ denotes the length of the shortest closed curve $\alpha$ satisfying $\int_{\alpha}\Xi\ne 0$ with respect to the metric $\tilde{g}_0$ and $g_0$ respectively. 

We can now apply the systolic inequality by Brendle-Hung \cite{Brendle-Hung} to $(\hat{M}_{\lambda},\tilde{g}_0)$, which implies that
\[
\tilde{H}_{\hat{\Sigma}_{\lambda}}-(n-1)\le \frac{1}{2}\left(\frac{4\pi}{\tilde{\sigma}n}\right)^{n}\le \frac{1}{2}\left(\frac{4\pi}{\hat{\sigma}n}\right)^n\left(1-\frac{\ve}{4}\right)^{-n}.
\]
These inequalities together imply that
\begin{align*}
\hat{\sigma}\lambda^{-1}\le \frac{1-\frac{\ve}{4}}{1-\frac{\ve}{2}}\sigma.
\end{align*}
Since $\ve$ is independent of $\lambda$ and $\delta$, by choosing $\lambda$ large and $\delta$ small, we obtain a contradiction.
\end{proof}
\begin{remark}
    One difference between the proof of Theorem \ref{thm:sharp-total-mean-curvature} and previous works (e.g. Theorem \ref{thm:Shi-Tam} and Theorem \ref{thm:integral-H-AF}) is that, instead of directly using the conclusion of the positive mass theorem, we utilize the sharp systolic inequality. It would be interesting to see if one can deduce the result directly by applying the positive mass theorem.
\end{remark}
\section{Proof of Theorem \ref{thm:interpolation}}
We fix two arbitrary Riemannian metrics $\hat{\gamma},\gamma$ on $\Sigma$ with $\gamma>\hat{\gamma}$. Define $Q=\gamma-{\hat{\gamma}}$ as a positive symmetric $(0,2)$-tensor on $\Sigma$. 

We consider the band $\Sigma\times [0,1]$ and denote $\Sigma_t:=\Sigma\times \{t\}$. Consider the Riemannian metric $\hat{g}$ on $\Sigma\times [0,1]$ of the form $\hat{g}=dt^2+\gamma_t$, where $\gamma_t$ is a family of Riemannian metric $\gamma_t=(1-t)\hat{\gamma}+t{\gamma}=\hat{\gamma}+tQ$. We denote $\hat{A}_{\Sigma_t}$ and $\hat{H}_{\Sigma_t}$ to be the second fundamental form and the mean curvature of $\Sigma_t$ with respect to the metric $\hat{g}$ and the unit normal vector $\frac{\p}{\p t}$. One can see that
\[
\hat{A}_{\Sigma_t}=\frac{1}{2}Q,\qquad \hat{H}_{\Sigma_t}=\frac{1}{2}\tr_{\gamma_t}Q.
\]

Let $v(x,t)$ be a smooth positive function on $\Sigma\times [0,1]$. We consider the Riemannian metric $g=\gamma_t+v^2dt^2$. We let $A_{\Sigma_t}$ and $H_{\Sigma_t}$ to be the second fundamental form and the mean curvature with respect to $g$ and the unit normal $v^{-1}\frac{\p}{\p t}$. It follows that $A_{\Sigma_t}=v^{-1}\hat{A}_{\Sigma_t}$ and $H_{\Sigma_t}=v^{-1}\hat{H}_{\Sigma_t}$.

We now summarize the result by Shi-Wang-Wei \cite{SWW}.
\begin{lemma}[\cite{SWW}]{\label{lem:band-metric}}
    For any smooth function $h>0$ defined on $\Sigma$, there exists some smooth function $v$ such that the Riemannian metric $g$ defined as above has the following properties:
    \begin{enumerate}[(i)]
        \item $g|_{\Sigma_0}=\hat{\gamma}$ and $g|_{\Sigma_1}=\gamma$.
        \item $H_{\Sigma_0}=h$.
        \item $H_{\Sigma_t}>0$ for all $t\in [0,1]$.
        \item $R_g=-K$ for some positive constant $K=K(\gamma,\hat{\gamma})$.
        \item We have $$
        \int_{\Sigma}H_{\Sigma_0}d\vol_{\hat{\gamma}}\le \int_{\Sigma}H_{\Sigma_1}d\vol_{\gamma} $$
    \end{enumerate}
\end{lemma}
\begin{proof}
    We describe the approach in \cite{SWW}. Consider the equation
    \begin{equation}{\label{eqn:interpolation-PDE}}
        \begin{cases}
            \frac{1}{2}\tr_{\gamma_t}Q\,\frac{\p v}{\p t}=v^2\lp_{\gamma_t}v-\frac{1}{2}(K+R_{\gamma_t})v^3+\frac{1}{2}(R_{\gamma_t}-R_{\hat{g}})v,\\
            v(\cdot,0)=v_0,
        \end{cases}
    \end{equation}
    where $v_0=\frac{1}{2}h^{-1}\tr_{\hat{\gamma}}Q>0$ and $K=\sup_{\Sigma\times [0,1]}|R_{\gamma_t}|$. We define
    \[
    N=\sup_{(x,t)\in\Sigma\times [0,1]}2\frac{\left|R_{\gamma_t}-R_{\hat{g}}\right| + |K+R_{\gamma_t}|}{\tr_{\gamma_t}Q},
    \]
    which depends only on $\gamma$ and $\hat{\gamma}$. By the maximum principle, we can compare $v$ to the solutions of the ODEs:
    \[
    \begin{cases}
        \frac{\p w_-}{\p t}=-\frac{N}{2}(w_-^3+w_-),\\
        w_-(0)=\min_{\Sigma} v_0,
    \end{cases},
    \qquad  
    \begin{cases}
        \frac{\p w_+}{\p t}=\frac{N}{2}w_+,\\
        w_+(0)=\max_{\Sigma}v_0.
    \end{cases}
    \]
    These ODEs can be solved explicitly over $[0,1]$:
    \[
    w_-(t)=\frac{\min v_0}{\left((1+\min v_0^2)e^{Nt}-\min v_0^2\right)^{\frac{1}{2}}},\qquad w_+(t)=\max v_0 e^{\frac{N}{2}t}.
    \]
    Therefore, there exists a smooth solution $v(x,t)$ of \eqref{eqn:interpolation-PDE} over $\Sigma\times [0,1]$ by standard parabolic analysis.
    
    Next, since $v$ solves \eqref{eqn:interpolation-PDE}, the metric $g=v^2dt^2+\gamma_t$ satisfies $H_{\Sigma_0}=h$ and $R_g=-K$. Clearly, we also have $g|_{\Sigma_0}=\hat{\gamma}$ and $g|_{\Sigma_1}=\gamma$. 
    
    To verify the estimate of the mean curvature, we compute the variation of the mean curvature along $\frac{\p}{\p t}$
    \begin{align*}
        &\frac{d}{dt}\int_{\Sigma_t}H_{\Sigma_t}d\vol_{\gamma_t}\\
        &=\int_{\Sigma_t}-\lp_{\Sigma_t}v- (|A_{\Sigma_t}|_{\gamma_t}^2+\Ric(\nu,\nu))v+H_{\Sigma_t}^2vd\vol_{\gamma_t}\\
        &=\frac{1}{2}\int_{\Sigma_t}(H_{\Sigma_t}^2-|A_{\Sigma_t}|^2_{\gamma_t})vd\vol_{\gamma_t}+\frac{1}{2}\int_{\Sigma_t}(R_{\gamma_t}-R_g)vd\vol_{\gamma_t}\\
        &=\frac{1}{2}\int_{\Sigma_t}\frac{1}{4}((\tr_{\gamma_t}Q)^2-|Q|_{\gamma_t}^2)v^{-1}d\vol_{\gamma_t}+\frac{1}{2}\int_{\Sigma_t}(R_{\gamma_t}+K)vd\vol_{\gamma_t}.
    \end{align*}
    Since $Q>0$, we have $(\tr_{\gamma_t}Q)^2-|Q|_{\gamma_t}^2>0$. Additionally, by the definition of $K$, we have $R_{\gamma_t}+K\ge 0$. Therefore, the quantity $\int_{\Sigma_t}H_{\Sigma_t}d\vol_{\gamma_t}$ is monotone increasing in $t$. In particular, we have 
    $$\int_{\Sigma_{0}}H_{\Sigma_0}d\vol_{\hat{\gamma}}< \int_{\Sigma_{1}}H_{\Sigma_1}d\vol_{\gamma}$$
    This completes the proof.
\end{proof}

Using Lemma \ref{lem:gluing}, we are able to deduce Theorem \ref{thm:interpolation}.
\begin{proof}[Proof of Theorem \ref{thm:interpolation}]
    Suppose $\hat{g}$ is a Riemannian metric on $\Omega=B^2\times T^{n-2}$ with $\hat{g}|_{\p\Omega}=\hat{\gamma}$ and $R_{\hat{g}}\ge-n(n-1)$. Let $H_{\hat{g}}>0$ denote the mean curvature of $\Sigma$ with respect respect to the metric $\hat{g}$. Given another Riemannian metric $\gamma$ on $\Sigma$ with $\gamma>\hat{\gamma}$, we take $h=\frac{1}{2}H_{\hat{g}}$. By Lemma \ref{lem:band-metric}, there exists a Riemannian metric $g'$ defined on $\Omega':=\Sigma\times [0,1]$ with the following properties:
    \begin{enumerate}[(i)]
        \item $g'|_{\Sigma\times \{0\}}=\hat{\gamma}$ and $g'|_{\Sigma\times \{1\}}=\gamma$.
        \item $H_{g'}=h=\frac{1}{2}H_{\hat{g}}$ on $\Sigma\times \{0\}$ and $H_{g'}>0$ on $\Sigma\times \{1\}$.
        \item $R_{g'}= -K$ for some positive constant $K=K(\hat{\gamma},\gamma)$.
        \item We have
        \begin{equation}{\label{eqn:mean-curvature-loss}}
        \int_{\Sigma} H_{g'}(x,0)d\vol_{\hat{\gamma}} \le \int_{\Sigma}H_{g'}(x,1)d\vol_{\gamma}.
        \end{equation}
    \end{enumerate}
    Note that $\hat{g}|_{\p\Omega}=g'|_{\Sigma\times\{0\} }$, and $H_{g'}<H_{\hat{g}}$ at each point $x\in\Sigma$. Applying Lemma \ref{lem:gluing} we can glue $(\Omega,\hat{g})$ with $(\Sigma\times [0,1],g')$ by identifying $\p\Omega$ and $\Sigma\times \{0\}$, resulting in a Riemannian metric $g$ defined on $B^2\times T^{n-2}$ (which is diffeomorphic to $\Omega$) with the following properties:
    \begin{itemize}
    \item The scalar curvature satisfies \( R_g \geq \min \{ -K, -n(n - 1) \} - 1 \).
    \item \( g|_{\partial \Omega} = \gamma \).
    \end{itemize}
    To check the condition on the mean curvature, equation \eqref{eqn:mean-curvature-loss} implies
    \[
   \frac{1}{2}\int_{\Sigma}H_{\hat{g}}d\vol_{\hat{\gamma}}\le \int_{\Sigma}H_{g'}(x,1)d\vol_{\gamma}=\int_{\Sigma}H_gd\vol_{\gamma}
    \]
    Finally, we can rescale \( g \) so that the scalar curvature is bounded from below by \( -n(n - 1) \). This completes the proof of Theorem \ref{thm:interpolation}.
\end{proof}
\appendix
\section{The Gluing Construction by Brendle-Marques-Neves}
We recall the gluing construction by Brendle-Marques-Neves \cite{BMN}. Similar types of constructions could be found in the works of Miao \cite{Miao} and B\"ar-Hanke \cite{BH-boundary}. We assume $X$ is some compact manifold without boundary.
\begin{lemma}[\cite{BMN}]{\label{lem:BMN-gluing}}
    Let $V=X\times [0,r)$ for some $r>0$. Suppose $g_1$ and $g_2$ be two Riemannian metric defined on $V$ such that $g_1-g_2=0$ at each point on $X\times \{0\}$. Moreover, we assume $H_{g_1}>H_{g_2}$ at each point on $X\times \{0\}$, where  $H_{g_i}$ denotes the mean curvature of $X\times \{0\}$ within $V$ with respect to $g_i$ and the outward unit normal vector.

    Then, given any $\ve>0$, there exists a Riemannian metric $g$ on $V$ with the following properties:
    \begin{itemize}
        \item $R_g\ge \min\{R_{g_1},R_{g_2}\}-\ve$ at each point in $V$.
        \item $g$ agree with $g_1$ on $X\times [\frac{r}{2},r)$.
        \item $g$ agree with $g_2$ on $X\times [0,\delta)$ for some $\delta$ depending on $\ve$.
    \end{itemize}
\end{lemma}
\begin{lemma}[cf. Theorem 3.7 in \cite{BH-boundary}]{\label{lem:gluing}}
    Let $V^-=X\times (-r,0]$ and $V^+=X\times [0,r)$ for some $r>0$. Suppose $g_{\pm}$ are Riemannian metrics on $V^{\pm}$. Let $H_{+}$ (resp. $H_-$) denote the mean curvature of $X\times \{0\}$ within $V^{+}$ (resp. $V^-$) with respect to the inward (resp. outward) unit normal vector. Suppose $g_+|_{X\times \{0\}}=g_-|_{X\times \{0\}}$ and $H_+<H_-$ at each point on $X$. Then, for any $\ve>0$, there exists some Riemannian metrics $\hat{g}_{\ve}$ defined on $X\times (-r,r)$ with:
    \begin{itemize}
        \item $R_{\hat{g}_{\ve}}\ge \min \{\inf_{V^-}R_{g_-},\inf_{V^+}R_{g_+}\}-\ve$.
        \item $\hat{g}_{\ve}$ agrees with $g_+$ on $X\times [\eta,r)$ and $\hat{g}_{\ve}$ agrees with $g_-$ on $X\times (-r,-\eta]$ for some $0<\eta \ll r$.
    \end{itemize}
\end{lemma}
\begin{proof}
Given $\ve>0$, we choose some $\eta>0$ small enough and extend $g_+$ smoothly to a metric $\Tilde{g}_+$ on $X\times (-\eta,r)$ with 
$$R_{\Tilde{g}_+}|_{(-\eta,0]}\ge \inf_{V^+}R_{g_+}-\frac{\ve}{2}.$$ 
Applying Lemma \ref{lem:BMN-gluing} on $X\times (-\eta,0]$ with the metrics $g_1=g_-|_{X\times (-\eta,0]}$ and $g_2=\tilde{g}_+|_{X\times (-\eta,0]}$ yields a Riemannian metric $g_{\ve}$ defined on $X\times (-\eta,0]$ with
\begin{itemize}
    \item $R_{g_{\ve}}\ge \min \{R_{g_-},R_{\tilde{g}_+}\}-\frac{\ve}{2}$ at each point in $X\times (-\ve,0]$.
    \item $g_{\ve}$ agrees with $\tilde{g}_+$ on $X\times (-\delta,0]$ for some $0<\delta\ll\eta$.
    \item $g_{\ve}$ agrees with $g_-$ on $X\times (-\eta,-\frac{\eta}{2}]$.
\end{itemize}
One can then glue the metric $g_-$ and $\tilde{g}_+$ together via $g_{\ve}$, result in a new metric $\hat{g}_{\ve}$ defined on $X\times (-r,r)$. One can then easily check that $\hat{g}_{\ve}$ satisfies all the desired properties.
\end{proof}

\end{document}